\documentclass[11pt,reqno,tbtags]{amsart}

\usepackage{amsthm,amssymb,amsfonts,amsmath}
\usepackage{amscd} 
\usepackage{mathrsfs}
\usepackage[numbers, sort&compress]{natbib} 
\usepackage{subcaption, graphicx}
\usepackage{vmargin}
\usepackage{setspace}
\usepackage{paralist}
\usepackage[pagebackref=true]{hyperref}
\usepackage{color}
\usepackage[normalem]{ulem}
\usepackage{stmaryrd} 
\usepackage[refpage,noprefix,intoc]{nomencl} 
\usepackage{tcolorbox}
\usepackage{bbm,mathtools,tikz}
\usepackage{MnSymbol,wasysym}

\providecommand{\Z}{}

\renewcommand{\Z}{\mathbb{Z}}

\newcommand{\I}[1]{{\mathbbm 1}_{#1}}

\newcommand\cB{\mathcal B}

\newcommand\cD{\mathcal D}

\newcommand\cU{{\mathcal U}}

\usepackage{crossreftools,booktabs}

\makeatletter
\newcommand{\optionaldesc}[2]{%
  \phantomsection
  #1\protected@edef\@currentlabel{#1}\label{#2}%
}
\makeatother

\newcommand{\pran}[1]{\left(#1\right)}

\providecommand{\ora}[1]{}
\renewcommand{\ora}[1]{\overrightarrow{#1}}
\DeclareRobustCommand{\SkipTocEntry}[5]{} 

\newcommand{\tree}{\mathrm{t}}
\newcommand{\tfree}{\mathrm{t}^{\mathrm{free}}}
\newcommand{\tfix}{\mathrm{t}^{\mathrm{fix}}}
\newcommand{\tspn}[1]{\tree^{\mathrm{spn}}_{#1}}

\renewcommand{\path}{\mathrm{d}}

\newcommand{\HS}{\bar{\mathcal{S}}}
\newcommand{\RHS}{\mathcal{S}}
\newcommand{\HSv}[1]{\RHS_{#1}}

\newcommand{\halfceil}[1]{\lceil\frac{#1}{2}\rceil}
\newcommand{\halffloor}[1]{\lfloor\frac{#1}{2}\rfloor}

\newtheorem{thm}{Theorem}
\newtheorem{lem}[thm]{Lemma}
\newtheorem{prop}[thm]{Proposition}
\newtheorem{cor}[thm]{Corollary}

\makenomenclature										
\graphicspath{ {./images/} }
\usepackage{wrapfig}

\numberwithin{equation}{section}
\numberwithin{thm}{section}

\newcommand{\vp}{\varphi}
\newcommand{\ints}[1]{\llbracket #1 \rrbracket}

\begin{document}
\date{June 5, 2024} 

\title{Refined Horton-Strahler numbers I: a discrete bijection} 
\author{Louigi Addario-Berry}
\address{Department of Mathematics and Statistics, McGill University, Montr\'eal, Canada}
\email{louigi.addario@mcgill.ca}

\author{Marie Albenque}
\address{Universit\'e Paris Cit\'e, CNRS, IRIF, F-75013, Paris, France}
\email{malbenque@irif.fr}

\author{Serte Donderwinkel}
\address{University of Groningen, Bernoulli Institute for Mathematics, Computer Science and AI, and CogniGron (Groningen Cognitive Systems and
Materials Center), Groningen, The Netherlands}
\email{s.a.donderwinkel@rug.nl}

\author{Robin Khanfir}
\address{Department of Mathematics and Statistics, McGill University, Montr\'eal, Canada}
\email{robin.khanfir@mcgill.ca}


\begin{abstract} 
The Horton--Strahler number of a rooted tree $T$ is the height of the tallest complete binary tree that can be homeomorphically embedded in $T$. The number of full binary trees with $n$ internal vertices and Horton-Strahler number $s$ is known \cite{FLAJOLET79} to be the same as the number of Dyck paths of length $2n$ whose height $h$ satisfies $\lfloor \log_2(1+h)\rfloor=s$. 

In this paper, we present a new bijective proof of the above result, that in fact strengthens and refines it as follows. We introduce a sequence of trees $(\tau_i,i \ge 0)$ which ``interpolates'' the complete binary trees, in the sense that $\tau_{2^h-1}$ is the complete binary tree of height $h$ for all $h \ge 0$, and $\tau_{i+1}$ strictly contains $\tau_i$ for all $i \ge 0$. Defining $\RHS(T)$ to be the largest $i$ for which $\tau_i$ can be homeomorphically embedded in $T$, we then show that the number of full binary trees $T$ with $n$ internal vertices and with $\RHS(T)=h$ is the same as the number of Dyck paths of length $2n$ with height $h$. (We call $\RHS(T)$ the {\em refined Horton--Strahler number} of $T$.)

Our proof is bijective and relies on a recursive decomposition of binary trees (resp.\ Dyck paths) into subtrees with strictly smaller refined Horton--Strahler number (resp.\ subpaths with strictly smaller height). In a subsequent paper \cite{US2}, we will show that the bijection has a continuum analogue, which transforms a Brownian continuum random tree into a Brownian excursion and under which (a continuous analogue of) the refined Horton--Strahler number of the  tree becomes the height of the excursion. 
\end{abstract} 
\maketitle

\section{Introduction}
The purpose of this paper is to introduce a quantity we call the {\em refined Horton--Strahler number}, and to use it to strengthen existing bijective results on binary trees with given Horton--Strahler number. In this section, we first define the Horton--Strahler number and provide a brief literature review, then define our refinement and present the contributions of this paper.

Let $\cU:=\{1,2\}^*=\{\emptyset\}\cup \bigcup_{\ell\geq 1} \{1,2\}^\ell$. For $u=u_1u_2\ldots u_\ell \in \cU$ we write $|u|=\ell$ for the length of $u$; we take $|\emptyset|=0$ by convention. The ancestors of $u$ are $\{u_1u_2\dots u_k:0\le k \le |u| \}$. For $u,v\in\cU$, we write both $uv$ and $u\star v$ to denote the concatenation of the word $u$ with the word $v$. For $V\subset \cU$ we also write $uV=u*V:=\{uv\, :\, v\in V\}$. 

A {\em binary tree} is a finite set $\tree\subset \cU$ with $\emptyset \in \tree$ such that for all $u=u_1\ldots u_\ell \in \cU$, if $u \in \tree$ then  $p(u):=u_1\ldots u_{\ell-1} \in \tree$ in $\tree$. The elements of $\tree$ are its {\em vertices.} A vertex $v \in \tree$ is {\em internal} if either $v1 \in \tree$ or $v2 \in \tree$; otherwise it is a {\em leaf}. A binary tree $\tree$ is {\em full} if for all $u \in \tree$, $u1 \in \tree$ if and only if $u2 \in \tree$. 

We write $\preceq$ for the ancestral partial order on $\cU$ and $\preceq_{\mathrm{lex}}$ for the lexicographic order on $\cU$. Also, for $u,v \in \cU$ we write $u \wedge v$ for the most recent common ancestor of $u$ and $v$ in~$\cU$. 

Given two trees $\tree,\tree'$, an injective function $\varphi:\tree \to \tree'$ is an {\em embedding} if it 
is strictly increasing with respect to $\preceq_{\mathrm{lex}}$ and additionally satisfies $\varphi(u \wedge v)=\varphi(u)\wedge \varphi(v)$  for all $u,v \in \tree$. We say $\tree$ can be embedded in $\tree'$ if and only if there exists an embedding $\varphi:\tree \to \tree'$. 

The {\em complete binary tree of height $s$} is the tree $\mathrm{cb}(s):=\{u \in \cU: 0 \le |u| \le s\}$. 
The {\em Horton--Strahler number} of a binary tree $\tree$ is defined as 
\[
\HS(\tree):= \max(s: \mathrm{cb}(s)\mbox{ can be embedded in }\tree).
\]
For example, the tree $\tree$ on the left of Figure~\ref{fig:exHS} has $\HS(\tree)=2$; the center of the figure depicts an embedding of $\mathrm{cb}(2)$ into $\tree$. 
\begin{figure}[htb]
    \centering
    \includegraphics[page=1,width=0.25\linewidth]{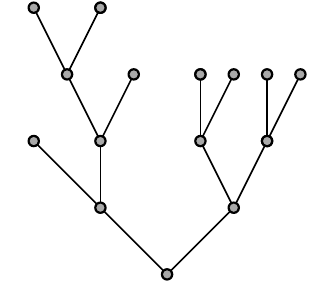}\qquad\qquad
    \includegraphics[page=2,width=0.25\linewidth]{images/HSrefined_binary.pdf}\qquad\qquad
    \includegraphics[page=3,width=0.25\linewidth]{images/HSrefined_binary.pdf}
    \caption{Example of the plane tree $\tree$, together with one embedding of $\mathrm{cb}(2)$ (middle) and one embedding of $\tau_5$ (right). Neither $\mathrm{cb}(3)$ nor $\tau_6$ can be embedded in $\tree$, so $\tree$ has Horton--Strahler number $\HS(\tree)=2$ and refined Horton--Strahler number $\RHS(\tree)=5$. The definition of $\tau_i$ is given later in the introduction and is illustrated in Figure~\ref{fig:hs_example}}.
    \label{fig:exHS}
\end{figure}

Originally introduced in hydrogeology by Horton~\cite{Hor45} in 1945 and Strahler~\cite{Str52} in 1952 to quantitatively study river systems, the Horton--Strahler number has since found many and varied fields of application. Among others areas, the Horton--Strahler number appears in physics, chemistry, biology, and social network analysis. It is also particularly useful in computer science, where it is called the \emph{register function}, and is used in  optimizing the manipulation of certain data structures. An overview of many of these  applications is given by Viennot~\cite{Vie90}. More specifically in mathematics, we refer to Esparza, Luttenberger \& Schlund~\cite{esparza} for a review of the occurrences of the Horton--Strahler number in combinatorics, algebra, logic, topology, and approximation theory. Finally, most of the previous probabilistic works on the topic were focused on the asymptotic behavior of the Horton--Strahler number of large random trees: see  Flajolet, Raoult \& Vuillemin \cite{FLAJOLET79}, Kemp \cite{KEMP79}, Devroye \& Kruszewski~\cite{Dev95}, Drmota \& Prodinger~\cite{Drm06},  Brandenberger, Devroye \& Reddad~\cite{BraDevRed21}, and Khanfir~\cite{Kha23,Kha24}. We also mention the works of Burd, Waymire \& Winn~\cite{Burd} and Kovchegov, Xu \& Zaliapin~\cite{KovZal20,KovZal21,KovXuZal23}, which study the Horton--Strahler number through the lens of dynamical systems theory.

In order to state our results,
it is convenient to introduce the notation $\ints{j,k}=\{j,j+1,\ldots,k\}$ for $j,k \in \Z_{\ge 0}$ with~$j \le k$. 
A {\em Dyck path of length $2n$} is a function 
$\mathrm{d}:\ints{0,2n} \to \Z_{\ge 0}$ with $\mathrm{d}(0)=\mathrm{d}(2n)=0$ and $|\mathrm{d}(i)-\mathrm{d}(i-1)|=1$ for all $i \in \{1,\ldots,2n\}$. The {\em height} of $\mathrm{d}$ is $\|\path \|:=\max(\path (i),i \in \ints{0,2n})$.
Write $\mathcal{D}(n)$ for the set of Dyck paths of length $2n$, and for $s \in \Z_{\ge 0}$ let 
\[
\overline{\cD}_{n,s}=\{\mathrm{d} \in \cD(n): \lfloor \log_2(1+\|\mathrm{d}\|)\rfloor=s\}\,
\]
be the set of Dyck paths with logarithmic height $s$. 
Also, let $\cB(n)$ be the set of full binary trees with $n$ internal vertices, and for $s \in \Z_{\ge 0}$ let $\overline{\cB}_{n,s}=\{\tree \in \cB(n): \HS(\tree)=s\}$ 
be the elements of $\cB(n)$ with Horton--Strahler number $s$. Flajolet, Raoult \& Vuillememin~\cite{FLAJOLET79} and Kemp~\cite{KEMP79} independently
 showed that for all $n \in \Z_{\ge 0}$ and $s \in\Z_{\ge 0}$,
\begin{equation}\label{eq:francon}
|\overline{\cD}_{n,s}|=|\overline{\cB}_{n,s}|;
\end{equation}
a bijective proof of this fact was found by Fran\c con \cite{franccon1984nombre} (see also Viennot~\cite{GERARDVIENNOT2002317} and Zeilberger~\cite{ZEILBERGER199089}). The fact that this identity involves the logarithmic height $\lfloor \log_2(1+\|\path \|)\rfloor$ suggests that a refinement of \eqref{eq:francon}, which uses the actual height, should exist; showing that this is indeed the case is the main goal of the paper.

The refinement of the set $\overline{\cD}_{n,s}$ is easy; for $h \in \Z_{\ge 0}$ let 
\[
\mathcal{D}_{n,h}=\{\mathrm{d}\in \cD(n):\|\mathrm{d}\|=h\},
\text{ so that }\quad\overline{\cD}_{n,s}=\bigcup_{h=2^s-1}^{2(2^s-1)} \mathcal{D}_{n,h}\, . 
\]

The refinement of the set $\overline{\cB}_{n,s}$ is slightly more involved than that of $\overline{\cD}_{n,s}$. Define the {\em refined Horton--Strahler number} as follows. Let $\tau_0=\emptyset$ consist of just a root vertex and let $\tau_1=\{\emptyset,1,2\}$ be the full binary tree of height $1$. Then, for $m\geq 1$, let $\tau_{2m}=\{\emptyset\}\cup (1\star \tau_{m})\cup (2\star \tau_{m-1})$ be the binary tree with $\tau_m$ rooted at the left child of the root and $\tau_{m-1}$ rooted at the right child of the root and, similarly, let $\tau_{2m+1}=\{\emptyset\}\cup (1\star \tau_{m})\cup (2 \star \tau_{m})$.  The trees $\tau_0,\ldots,\tau_7$ are depicted in Figure~\ref{fig:hs_example}. It is straightforward to see that for all $r$ it holds that $\tau_r\subset \tau_{r+1}$, and that $\tau_{2^s-1}=\mathrm{cb}(s)$ for all $s \in \Z_{\ge 0}$ 
\begin{figure}[htb]
    \centering
    \includegraphics[width=0.7\textwidth]{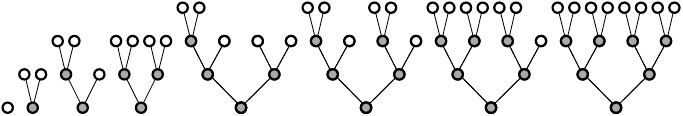}
    \caption{The trees $\tau_0,\ldots,\tau_7$ are depicted in order from left to right.}
\label{fig:hs_example}
\end{figure}

We define
\[ 
\RHS(\tree):=\max\{ r: \tau_r  \text{ can be embedded in }\tree\}.
\]
For example, the tree $\tree$ on the left of Figure~\ref{fig:exHS} has $\RHS(\tree)=5$; a depiction of an embedding of $\tau_5$ in $\tree$ appears on the right of the figure.
For $n \in \Z_{\ge 0}$ and $h \in \Z_{\ge 0}$, let  
\[
\cB_{n,h}=\{\tree \in \cB(n): \RHS(\tree)=h\}\, ,
\]
Note that for $s \ge 0$, the trees $\tau_{2^s-1},\ldots,\tau_{2(2^s-1)}$ all contain $\mathrm{cb}(s)$ and are strictly contained in $\mathrm{cb}(s+1)$, so 
\[
\overline{\cB}_{n,s} = \bigcup_{h=2^s-1}^{2(2^s-1)} \cB_{n,h}\, .
\]
The following is our main result.

\begin{thm}\label{thm:main}
 For all integers $n \ge 0$ and $h \ge 0$,  $|\cB_{n,h}|=|\cD_{n,h}|$.
\end{thm}

The identity \eqref{eq:francon} follows by taking a union over $h \in \ints{2^s-1,2(2^s-1)}$. Our proof of Theorem~\ref{thm:main} appears in Section~\ref{sec:bij0}. (Section~\ref{sec:intro_refined_strahler} presents a recursive decomposition of full binary trees that is used in the proof.)

Our proof is bijective, and like Fran\c con's bijective proof of \eqref{eq:francon}, our bijection is recursively constructed. However, our bijection is not equivalent to that of Fran\c con --- this can be verified by considering the images of the Dyck path 01234543210 under the two bijections. It would be interesting to know whether Fran\c con's bijection can also be refined in order to prove Theorem~\ref{thm:main}. It would also be very interesting to find a non-recursive bijection between the sets $\cB_{n,h}$ and $\cD_{n,h}$.

As mentioned in the abstract, in a subsequent paper in preparation \cite{US2}, we will describe a continuous analogue of the bijection we use to prove Theorem~\ref{thm:main}, which relates Brownian continuum random trees of a given (continuous) Horton-Strahler number $h$ to Brownian excursions of height $h$.





We conclude the introduction by defining some notation that will be useful in the sequel.
For a binary tree $\tree$ and $u \in \tree$, we let $\theta_u \tree$ be the subtree of $\tree$ rooted at $u$; formally 
\[\theta_u\tree=\{v:uv\in \tree\}.\]
We will sometimes write $\RHS_u(\tree):=\RHS(\theta_u\tree)$ for the refined Horton--Strahler number of $\theta_u\tree$.


\section{A spinal decomposition of binary trees}\label{sec:intro_refined_strahler}

\begin{prop}\label{prop:recursive_RHS}
The refined Horton--Strahler number $\RHS{(\tree)}$ of a full binary tree $\tree$ has the following recursive definition.   
\begin{enumerate}
\item[$(a)$] $\RHS(\tree)=0$ if and only $\tree=\{\emptyset\}$.
\item[$(b)$] Otherwise,
\[
\RHS(\tree)=
\max\left(\RHS_1(\tree),\RHS_2(\tree),2\min\big(\HSv{1}(\tree),\HSv{2}(\tree)\big)+\I{\HSv{1}(\tree)>\HSv{2}(\tree)}+1\right)\, .
\]
\end{enumerate}
\end{prop}
\begin{proof}
  Statement (a) is immediate from the definition of the refined Horton--Strahler number. 
  For statement (b), note that an embedding $\vp:\tau_r \to \tree$ must satisfy one of the following three conditions:
  \begin{enumerate}
  \item $1$ is an ancestor of $\vp(\emptyset)$;
  \item $2$ is an ancestor of $\vp(\emptyset)$;
  \item $\vp(\emptyset)=\emptyset$.
  \end{enumerate}
   Observe from the definition of embeddings that $\vp(\emptyset)\preceq \vp(u)$ for all $u\in \tau_r$, so if (1) holds then setting $1*\hat{\vp}(u)=\vp(u)$ for all $u\in \tau_r$ defines a function $\hat{\vp}:\tau_r\to\theta_1 t$. It is then easy to check that $\hat{\vp}$ is an embedding. Conversely, if $\tilde{\vp}:\tau_r\to\theta_1 t$ is an embedding then $u\in \tau_r\mapsto 1*\tilde{\vp}(u)\in t$ is also one. Thus, the maximum value of $r$ for which $\vp$ can be chosen so that (1) (resp.\ (2)) holds is $\RHS_1(\tree)$ (resp.\ $\RHS_2(\tree)$). 
   
  
  If $\vp$ satisfies (3), then necessarily $1$ is an ancestor of $\vp(1)$ and $2$ is an ancestor of $\vp(2)$. It follows that
  $\vp|_{1*\theta_1\tau_r}$ induces an embedding of $\theta_1\tau_r$ into $\theta_1\tree$ and 
  $\vp|_{2\star \theta_2\tau_r}$ induces an embedding of $\theta_2\tau_r$ into $\theta_2\tree$. Since 
  $\theta_1\tau_r=\tau_{\lfloor r/2\rfloor}$ and $\theta_2\tau_r=\tau_{\lfloor (r-1)/2\rfloor}$, an embedding $\vp$ of $\tau_r$ can thus be chosen to satisfy (3) if and only if $\lfloor r/2\rfloor \le \RHS_1(\tree)$ and $\lfloor (r-1)/2\rfloor \le \RHS_2(\tree)$; this is equivalent to requiring that $r \le 2\min(\RHS_1(\tree),\RHS_2(\tree))+\I{\RHS_1(\tree)>\RHS_2(\tree)}+1$. The result follows.  
 \end{proof}
For any full binary tree $\tree$ and any internal vertex $v \in \tree$, applying Proposition~\ref{prop:recursive_RHS} to the subtree $\theta_v\tree$ yields that 
\begin{equation}\label{eq:recursive_bd}
\HSv{v}(\tree)
=
\max\big(\HSv{v1}(\tree),\HSv{v2}(\tree),2\min\big(\HSv{v1}(\tree),\HSv{v2}(\tree)\big)+\I{\HSv{v1}(\tree)>\HSv{v2}(\tree)}+1\big)\, ,
\end{equation}
an identity which we will shortly use. 

The definitions of the coming paragraph are depicted in Figure~\ref{fig:decompoTrees}. 
For a full binary tree $\tree\ne \{\varnothing\}$, write $m=m(\tree)=\lfloor \RHS(\tree)/2\rfloor$, write $u=u(\tree)$ for the $\preceq_{\mathrm{lex}}$-maximal vertex of $\tree$ for which $\HSv{u}(\tree)=\RHS(\tree)$, and let $\ell=\ell(\tree)=|u|$, so that $u=u_1\ldots u_\ell\in \{1,2\}^{\ell}$. 
For all $1\leq j\leq \ell$, set $v_j=v_j(\tree)=3-u_j=1+\I{u_j=1}$ and define
\[\tree_j^{\mathrm{spn}}=\theta_{u_1\ldots u_{j-1}v_j} \tree,\]
so that, informally, $\tspn{1},\dots,\tspn{\ell}$ are the trees hanging off the path from the root to $u$, and $v_1,\dots, v_\ell$ encode whether each of these trees hangs off the path on the left or the right. If $\RHS(\tree)\ge 1$ then $u$ is an internal vertex, and in this case we define
\[
\tfree
=
\begin{cases}
    \theta_{u1}\tree &\mbox{ if $\RHS(\tree)$ is even}\\
    \theta_{u2}\tree &\mbox{ if $\RHS(\tree)$ is odd,}
\end{cases} 
\quad \mbox{ and } \quad
\tfix=
\begin{cases}
    \theta_{u2}\tree &\mbox{ if $\RHS(\tree)$ is even}\\
    \theta_{u1}\tree &\mbox{ if $\RHS(\tree)$ is odd.}
\end{cases}
\]

\begin{figure}[hbt]
    \centering
    \subcaptionbox{Embedding of $\tau_2$ rooted at $u$}[0.28\linewidth]{\includegraphics[page=1,width=0.9\linewidth]{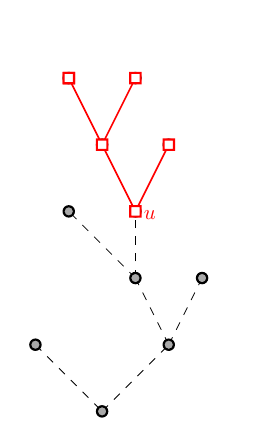}}\qquad
     \subcaptionbox{Value of $\RHS(\theta_v\tree)$ for all $v\in \tree$.}[0.28\linewidth]{\includegraphics[page=2,width=0.9\linewidth]{images/DecompositionTrees.pdf}}\qquad
     \subcaptionbox{The decomposition of $\tree$ along the ancestral path of~$u$.}[0.28\linewidth]{\includegraphics[page=3,width=0.9\linewidth]{images/DecompositionTrees.pdf}}
    \caption{Decomposition of a binary tree $\tree$ with $\RHS(\tree)=2$. This tree has $m(\tree)=1$, $u(\tree)=212$, $\ell(\tree)=|u|=3$ and $v_1,v_2,v_3=1,2,1$.}
    \label{fig:decompoTrees}
\end{figure}

Proposition~\ref{prop:recursive_RHS} implies the following result.
\begin{cor}\label{cor:maxRHS}
Let $\tree \ne \{\emptyset\}$ be a full binary tree. Then the vertices in $P:=\{v\in \tree : \HSv v(\tree)=\RHS(\tree)\}$ are precisely the ancestors of $u(\tree)$, and for any $1\leq j\leq \ell$, we have
\begin{equation}\label{eq:IneqAlongPath}
\RHS(\tree_j^{\mathrm{spn}})\le \frac{\RHS(\tree) -v_j}{2}.
\end{equation}
Moreover,
\begin{equation}\label{eq:freeFix}
    \RHS(\tree^{\mathrm{fix}})=\lceil \RHS(\tree)/2\rceil -1 =\lfloor (\RHS(\tree)-1)/2\rfloor\, ,
\end{equation}
and
\begin{equation}\label{eq:tfreebd}
\lfloor \RHS(\tree)/2\rfloor \le \RHS(\tfree)\leq \RHS(\tree)-1\, .
\end{equation}
\end{cor}
\begin{proof}
From the definition of the refined Horton--Strahler number, we have that $\HSv{v}(\tree)\geq \HSv{w}(\tree)$, whenever $v\preceq w$. This implies that the set of vertices $P:=\{v\in \tree : \HSv v(\tree)=\RHS(\tree)\}$ is a connected set of $\tree$ containing the path from $\emptyset$ to $u$. To see that $P$ does not contain any other vertex, we prove~\eqref{eq:IneqAlongPath}, which immediately implies that $P$ cannot contain any vertex which is not an ancestor of $u$.

Fix $1\leq j\leq \ell$, and write $v:=u_1\ldots u_{j-1}$; then \eqref{eq:recursive_bd} implies that 
\[
\RHS(\tree)=\HSv{v}(\tree) \ge 2\min\big(\HSv{v1}(\tree),\HSv{v2}(\tree)\big)+\I{\HSv{v1}(\tree)>\HSv{v2}(\tree)}+1.
\]
If $u_j=2$ then $\RHS_{v2}(\tree) \ge \RHS_{v1}(\tree)=\RHS(\tree_j^{\mathrm{spn}})$ so the above bound yields that 
\[
\RHS(\tree) \ge 2 \RHS(\tree_j^{\mathrm{spn}})+1=2 \RHS(\tree_j^{\mathrm{spn}})+v_j\, .
\]
If $u_j=1$ then $\RHS_{v1}(\tree) \ge \RHS_{v2}(\tree)=\RHS(\tree_j^{\mathrm{spn}})$, 
so $\RHS(\tree) \ge 2 \RHS(\tree_j^{\mathrm{spn}})+\I{\RHS_{v1}(\tree) > \RHS_{v2}(\tree)}+1$. 
This implies that $\RHS(\tree_j^{\mathrm{spn}}) < \RHS(\tree)=\RHS_{v1}(\tree)$, so in this case $\I{\RHS_{v1}(\tree) > \RHS_{v2}(\tree)}=1$ and we obtain 
\[
\RHS(\tree) \ge 2 \RHS(\tree_j^{\mathrm{spn}})+2=2 \RHS(\tree_j^{\mathrm{spn}})+v_j\, .
\]
The two preceding displays together establish \eqref{eq:IneqAlongPath}.

By $\preceq_{\mathrm{lex}}$-maximality of $u$, note that $\RHS(\tfree)\leq \RHS(\tree)-1$. To show (\ref{eq:freeFix}) and to complete the proof of (\ref{eq:tfreebd}), observe that since $\tree \ne \emptyset$, $\RHS(\tree)\neq 0$, so $u$ cannot be a leaf, and $u1,u2\in \tree$. By the definition of $u$, we have that $\HSv{u1}(\tree),\HSv{u2}(\tree)<\HSv{u}(\tree)=\RHS(\tree)$. By \eqref{eq:recursive_bd}, this implies that
\begin{equation}\label{eq:stu_ident}
\RHS(\tree)=2\min\big(\HSv{u1}(\tree),\HSv{u2}(\tree)\big)+\I{\HSv{u1}(\tree)>\HSv{u2}(\tree)}+1,
\end{equation}
so that $\RHS(\tree)$ is even if and only if $\HSv{u1}(\tree)>\HSv{u2}(\tree)$. We conclude with a case analysis. If $\RHS(\tree)$ is even, then $\tfree=\theta_{u1}\tree$ and $\tfix=\theta_{u2}\tree$, so $\RHS(\tfree)\geq 1+\RHS(\tfix)$ and \eqref{eq:stu_ident} becomes $\RHS(\tree)=2\RHS(\tfix)+2$. On the other hand, if $\RHS(\tree)$ is odd, then $\tfree=\theta_{u2}\tree$ and $\tfix=\theta_{u1}\tree$, so $\RHS(\tfix)\leq\RHS(\tfree)$ and \eqref{eq:stu_ident} becomes $\RHS(\tree)=2\RHS(\tfix)+1$. In both cases, we see that $\RHS(\tfix)=\lceil \RHS(\tree)/2\rceil-1$ and $\RHS(\tfree) \ge \lfloor \RHS(\tree)/2\rfloor$, which completes the proof.

\end{proof}



In what follows, for $h \in \Z_{\ge 0}$ we write $\cB_h=\bigcup_{n \ge 0} \cB_{n,h}$ for the set of full binary trees $\tree$ with $\RHS(\tree)=h$.

\begin{lem}\label{lem_tree:one_step_recursion}
For all $h \ge 1$, the function $G$ given by
\[
\tree\stackrel{G}{\longmapsto}\big (\tfix,\tfree,((v_j(\tree),\tree_j^{\mathrm{spn}}) ; 1\leq j\leq \ell(\tree))\big) 
\]
\noindent
is a bijection from $\mathcal{B}_h$ to the set
\[\mathcal{B}_h':=\mathcal{B}_{\halfceil{h}-1}\ \times\ \bigcup_{k=\halffloor{h}}^{h-1}\mathcal{B}_k \ \times \ \bigcup_{\ell\geq 0}\Big(\big(\{1\}\times\bigcup_{k=0}^{\halfceil{h}-1}\mathcal{B}_k\big)\,\cup\, \big(\{2\}\times\bigcup_{k=0}^{\halffloor{h}-1}\mathcal{B}_k\big)\Big)^\ell.\]

\end{lem}
\begin{proof}
Let $\tree\in\cB_h$ with $h\geq 1$. We first show that $G(\tree)\in\cB_h'$.  Corollary~\ref{cor:maxRHS}
implies that $\tree^{\mathrm{fix}} \in \cB_{\lceil \frac h2\rceil-1}$. Moreover, noting that $\lfloor (h-1)/2\rfloor=\lceil h/2\rceil-1$ and $\lfloor (h-2)/2\rfloor=\lfloor h/2 \rfloor-1$, 
the corollary also implies that 
\[
(v_j(\tree),\tree_j^{\mathrm{spn}}) \in 
\big(\{1\}\times\bigcup_{k=0}^{\halfceil{h}-1}\mathcal{B}_k\big)\,\cup\, \big(\{2\}\times\bigcup_{k=0}^{\halffloor{h}-1}\mathcal{B}_k\big)\, ,
\]
for $1 \le j \le \ell(\tree)$.
Finally, \eqref{eq:tfreebd} implies that $\RHS(\tfree)\in\{\halffloor{h},\ldots,h-1\}$, and thus $G(\tree)\in \cB_h'$. 

We will now show that $G$ is a bijection. Given $\big (\mathtt{t}^{\mathrm{fix}},\mathtt{t}^{\mathrm{free}},((\mathtt{v}_j,\mathtt{t}_j^{\mathrm{spn}}) ; 1\leq j\leq l)\big) 
 \in \mathcal{B}_h'$, we set $\mathtt{u}_j=3-\mathtt{v}_j$ for all $1\leq j\leq l$. Moreover, we set $\mathtt{u}=\mathtt{u}_1\ldots\mathtt{u}_l$ and then set 
\[
\mathtt{u}^{\mathrm{free}} =
\begin{cases}
    \mathtt{u}1 &\text{ if }h \text{ is even},\\
    \mathtt{u}2 &\text{ if }h \text{ is odd}
\end{cases}
\qquad 
\text{ and }
\qquad 
\mathtt{u}^{\mathrm{fix}} =
\begin{cases}
    \mathtt{u}2 &\text{ if }h \text{ is even},\\
    \mathtt{u}1 &\text{ if }h \text{ is odd.}
\end{cases}
\]
Also, we set
\begin{equation}
\label{inverse_G}
\mathtt{t}:=\{(\mathtt{u}_1,\ldots,\mathtt{u}_j) : 0\leq j\leq l\}\ \cup\  \mathtt{u}^{\mathrm{free}}*\mathtt{t}^{\mathrm{free}}\ \cup \ \mathtt{u}^{\mathrm{fix}}*\mathtt{t}^{\mathrm{fix}}\ \cup\  \bigcup_{j=1}^l (\mathtt{u}_1,\ldots,\mathtt{u}_{j-1},\mathtt{v}_j)*\mathtt{t}_j^{\mathrm{spn}}\, ,
\end{equation}
where we have written, e.g., $(\mathtt{u}_1,\ldots,\mathtt{u}_j)=\mathtt{u}_1\ldots \mathtt{u}_j$ for clarity. The expression (\ref{inverse_G}) defines a map $\widetilde{G}$ from $\mathcal{B}_h'$ to the set of full binary trees. We observe that, by construction, $\widetilde{G}\circ G(\tree)=\tree$ for any tree $\tree\in\mathcal{B}_h$. To complete the proof that $\widetilde{G}$ is the inverse of $G$, we only need to show that $\RHS(\mathtt{t})=h$, and that $\mathtt{u}$ is the $\preceq_{\mathrm{lex}}$-maximal vertex of $\mathtt{t}$ such that $\HSv{\mathtt{u}}(\mathtt{t})=h$, since this readily yields that $G\circ\widetilde{G}$ is the identity on $\mathcal{B}_h'$.

First, it follows from Proposition~\ref{prop:recursive_RHS} and the definitions of $\mathtt{u}^{\mathrm{free}}$ and $\mathtt{u}^{\mathrm{fix}}$ and of $\cB'_h$ , that $\HSv{\mathtt{u}}(\mathtt{t})=h$, and that $\HSv{\mathtt{u}1}(\mathtt{t}),\HSv{\mathtt{u}2}(\mathtt{t})<h$. Then, again using Proposition~\ref{prop:recursive_RHS}, it follows by decreasing induction on $k$ that $\HSv{\mathtt{u}_1\ldots \mathtt{u}_k}(\mathtt{t})=h$, for any $k\in \{0,\ldots,l\}$. This implies in particular that $\RHS(\mathtt{t})=h$. Finally, the fact that $\mathtt{u}$ is the lexicographically maximal vertex with $\RHS_{\mathtt{u}}(\mathtt{t})=h$ is immediate from the facts that $\RHS(\mathtt{t}_j^{\mathrm{spn}})<h$ for all $1 \le j \le l$ and that $\HSv{\mathtt{u}1}(\mathtt{t}),\HSv{\mathtt{u}2}(\mathtt{t})<h$. 
\end{proof}

\section{The bijection}\label{sec:bij0}

\begin{figure}[hbt]
    \centering
    \includegraphics[width=\textwidth]{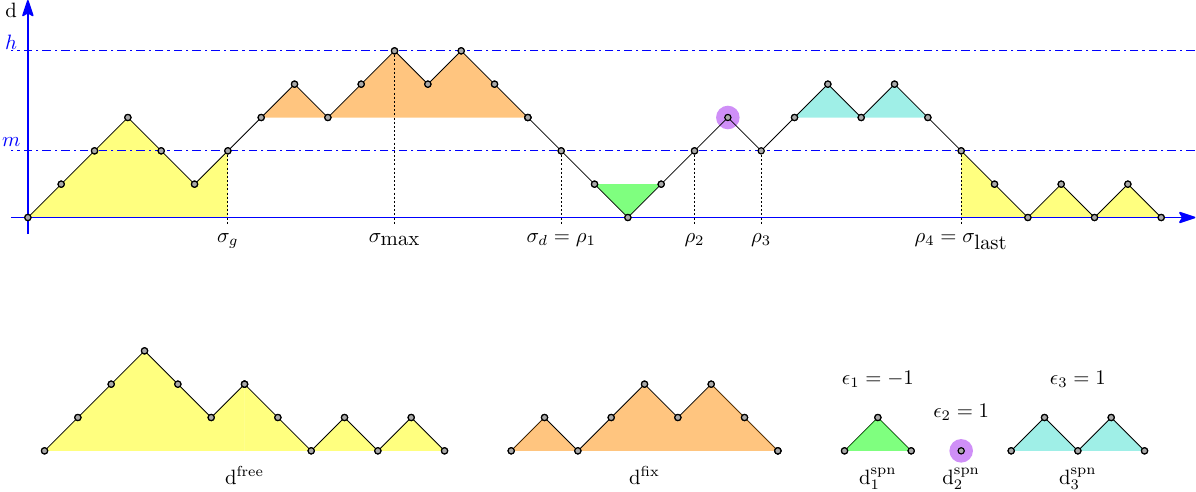}
    \caption{
    \emph{Top:} 
    A Dyck path $\mathrm{d}$, with $h=\|\path\|=5$, 
    so that $m=2$. The quantities $\sigma_g$, $\sigma_{\max}$, $\sigma_d$, 
     $\sigma_{\mathrm{last}}$, and $(\rho_i)$ are represented.\\
     \emph{Bottom:} 
     The paths $\path^{\mathrm{free}}$, $\path^{\mathrm{fix}}$ and $(\path^{\mathrm{spn}}_i)_{1\leq i \leq \ell}$ 
     associated to $\path$, together with the corresponding values of $(\epsilon_i)_{1\leq i \leq \ell}$.
     }
    \label{fig:sigmas}
\end{figure}

Write $\mathcal{D}_h=\bigcup_{n \ge 0} \cD_{n,h}$ for the set of all Dyck paths with height  $h$, and let
 $\cD=\bigcup_{n \ge 0} \cD(n)=\bigcup_{h \ge 0} \cD_h$ be the set of all Dyck paths. Likewise, let $\cB=\bigcup_{n \ge 0} \cB(n)=\bigcup_{h \ge 0} \cB_h$ be the set of all full binary trees. 
In this section we define a function $\Phi:\cD \to \cB$ and show that it induces a bijection between $\cD_{n,h}$ and $\cB_{n,h}$ for all $n,h \ge 0$, thereby proving Theorem~\ref{thm:main}. 

The definitions of the coming two paragraphs are depicted in Figure~\ref{fig:sigmas}. 
First, for $\path \in \cD_{n,h}$ with $h\geq 1$ write  $m=m(\mathrm{d})=\lfloor h/2\rfloor$, and let 
\begin{align*}
    \sigma_{\max} & =\sigma_{\max}(\mathrm{d})=\inf\{i \in \ints{0,2n}: \mathrm{d}(i)=h\}\, ,\\
    \sigma_g&=\sigma_g(\mathrm{d})=\sup\{i \in \ints{0,\sigma_{\max}}: \mathrm{d}(i)=m\}\, ,\\
    \sigma_d &=\sigma_d(\mathrm{d})=\inf\{i \in \ints{\sigma_{\max},2n} : \mathrm{d}(i)=m\}~\mbox{and}\\
    \sigma_{\mathrm{last}}
    & = \sigma_{\mathrm{last}}(\mathrm{d}) 
    = \sup\{i \in \ints{\sigma_{\max},2n}: \mathrm{d}(i)=m\}\, .
\end{align*}
Then list the elements of the set $\{i \in \ints{\sigma_d,\sigma_{\mathrm{last}}}:\mathrm{d}(i)=m\}$ in increasing order as 
\[
\sigma_d=\rho_1 < \ldots < \rho_\ell  < \rho_{\ell+1}=\sigma_{\mathrm{last}}\, ,
\]
so $\ell=\ell(\mathrm{d})=|\{i \in \ints{\sigma_d,\sigma_{\mathrm{last}}}:\mathrm{d}(i)=m\}|-1$. Note that if $\sigma_d=\sigma_{\mathrm{last}}$ then $\ell=0$ and $\rho_1=\sigma_{\mathrm{last}}$.

Next, still for $\path \in \cB_{n,h}$, define Dyck paths $\path^{\mathrm{fix}}$ of length $\sigma_d-\sigma_g-2$, $\path^{\mathrm{free}}$ of length $2n-(\sigma_{\mathrm{last}}-\sigma_g)$, and $\path_j^{\mathrm{spn}}$ of length $\rho_{j+1}-\rho_j-2$ for $j \in \ints{1,\ell}$ as follows:
\begin{align*}
\path^{\mathrm{fix}}(i)&=\path(\sigma_g+1+i)-m-1\text{ for }i \in \ints{0,\sigma_d-\sigma_g-2};\\
\path^{\mathrm{free}}(i)&= \begin{cases}
	\path(i)&\text{ for }i\in\ints{0,\sigma_g},\\
	\path(\sigma_{\mathrm{last}}+i-\sigma_g)&\text{ for }i\in \ints{\sigma_{g}+1,2n-(\sigma_{\mathrm{last}}-\sigma_g)}
\end{cases};\text{ and }\\
 \path_j^{\mathrm{spn}}(i)&=|\path(\rho_j+1+i)-m|-1
 \text{ for }i \in \ints{0,\rho_{j+1}-\rho_j-2}.
 \end{align*}
Setting
\[
\epsilon_j=\epsilon_j(\path)=\path(\rho_j+1)-\path(\rho_j)\, ,
\]
for $j \in \ints{1,\ell}$, then we may also write 
$\path_j^{\mathrm{spn}}(i)=\epsilon_j(\path(\rho_j+1+i)-m)-1$.  Finally, write $u_j=(3+\epsilon_j)/2$ and $v_j=(3-\epsilon_j)/2$, so that if $\epsilon_j=-1$ then $u_j=1$ and $v_j=2$, whereas if $\epsilon_j=1$ then $u_j=2$ and $v_j=1$.

We immediately record a few facts about the Dyck paths we just defined, as they will be useful in the sequel.  
\begin{lem}\label{lem:path_obvs}
If $\path$ is a Dyck path with $\|\path\|\ge 1$, then for all $j \in \ints{1,\ell}$, it holds that
\begin{equation*}
\|\path_j^{\mathrm{spn}}\|\le \frac{\|\path\| -v_j}{2}.
\end{equation*}
Moreover,
\begin{equation*}
    \|\path^{\mathrm{fix}}\|=\lceil\,  \|\path\|/2\, \rceil -1 =\lfloor (\|\path\|-1)/2\rfloor
\end{equation*}
and
\[
\lfloor \|\path\|/2\rfloor \leq \|\path^{\mathrm{free}}\|\leq \|\path\|-1\, .
\]
\end{lem}
\begin{proof}
    First, $\path^{\mathrm{fix}}(\sigma_{\max}-\sigma_g-1)=\path(\sigma_{\max})-m-1=h-m-1$, and $\path^{\mathrm{fix}}(i) \le \|\path\|-m-1=h-m-1$ for all $i$, so $\|\path^{\mathrm{fix}}\| =h-m -1$. Next, since $\path(\sigma_g)=m$ and $\path(k)< m$ for $k >\sigma_{\mathrm{last}}$, we have 
\[
\max( \path^{\mathrm{free}}(i),i\in \ints{0,2n-(\sigma_{\mathrm{last}}-\sigma_g)})=\max(\path(i),i \in \ints{0,\sigma_g}) \in \ints{m,h-1}\, ,
\]
so $m\leq \|\path^{\mathrm{free}}\|\leq h-1$. Finally, if $\epsilon_j=1$ then $m+1 \le \path(\rho_j+1+i) \le h$ for all $i \in \ints{0,\rho_{j+1}-\rho_j-2}$, so $\|\path_j^{\mathrm{spn}}\|\le h-(m+1)$; and likewise, if $\epsilon_j=-1$ then  $0 \le \path(\rho_j+1+i) \le m-1$ for all $i \in \ints{0,\rho_{j+1}-\rho_j-2}$, so $\|\path_j^{\mathrm{spn}}\|\le m-1$. Recalling that $\|\path\|=h$ and noting that $m=\lfloor h/2\rfloor$ and $h-m-1=\lceil h/2\rceil-1=\lfloor (h-1)/2\rfloor$ conclude the proof.
\end{proof}
We are now ready to describe the bijection $\Phi:\cD \to \cB$, which is illustrated in Figure~\ref{fig:bijection}. For $\path$ the unique Dyck path of length $0$, we let $\Phi(\path)=\{\emptyset\}=\tau_0$. 
For other elements of $\cD$, the bijection is recursively constructed, and may be informally described as follows. 
First, $\Phi(\path)$ contains the spine $\{u_1\ldots u_j : 0\leq j\leq \ell\}$; by convention $u_1\ldots u_j=\emptyset$ when $j=0$. Next, at the end of the spine, the tree $\Phi(\path^{\mathrm{free}})$ is rooted on the left if $h$ is even, and on the right otherwise; the tree $\Phi(\path^{\mathrm{fix}})$ is rooted on the right if $h$ is even, and on the left otherwise. Finally, for each $j \in \ints{1,\ell}$, the tree $\Phi(\path_j^{\mathrm{spn}})$ is rooted at a child of $u_1\ldots u_{j-1}$: on the left if $\epsilon_j=1$, and on the right if $\epsilon_j=-1$.

Formally, if $h=h(\path)$ is even then let 
\begin{multline*}
\Phi(\path)=\{u_1 \ldots u_j : 0\leq j\leq \ell\}\ \cup\  (u_1,\ldots,u_\ell,1)*\Phi(\path^{\mathrm{free}})\\
\cup \ (u_1,\ldots,u_\ell,2)*\Phi(\path^{\mathrm{fix}})\ \cup\  \bigcup_{j=1}^\ell (u_1,\ldots,u_{j-1},v_j)*\Phi(\path_j^{\mathrm{spn}})\, ,
\end{multline*}
and if $h$ is odd then let 
\begin{multline*}
\Phi(\path)=\{ u_1 \ldots u_j : 0\leq j\leq \ell\}\ \cup \ (u_1,\ldots,u_\ell,1)*\Phi(\path^{\mathrm{fix}})\\
\cup\ \ (u_1,\ldots,u_\ell,2)*\Phi(\path^{\mathrm{free}}) \ \cup\   \bigcup_{j=1}^\ell (u_1,\ldots,u_{j-1},v_j)*\Phi(\path_j^{\mathrm{spn}})\, ;
\end{multline*}
in the last two displays we have written $(u_1,\ldots,u_\ell,1)=u_1\ldots u_\ell 1$, and the like, to make the expressions easier to parse. 
Note that if $\path \in \cD(n)$ for $n > 0$ then the Dyck paths $\path^{\mathrm{fix}},\path^{\mathrm{free}}$, and $(\path^{\mathrm{spn}}_j, j \in \ints{1,\ell})$ all belong to $\bigcup_{m \in \ints{0,n-1}} \cD(m)$, so the above identities indeed recursively define $\Phi(\path)$ for all $\path \in \cD$.

\begin{figure}[t]
    \centering
    \includegraphics[width=\textwidth,page=2]{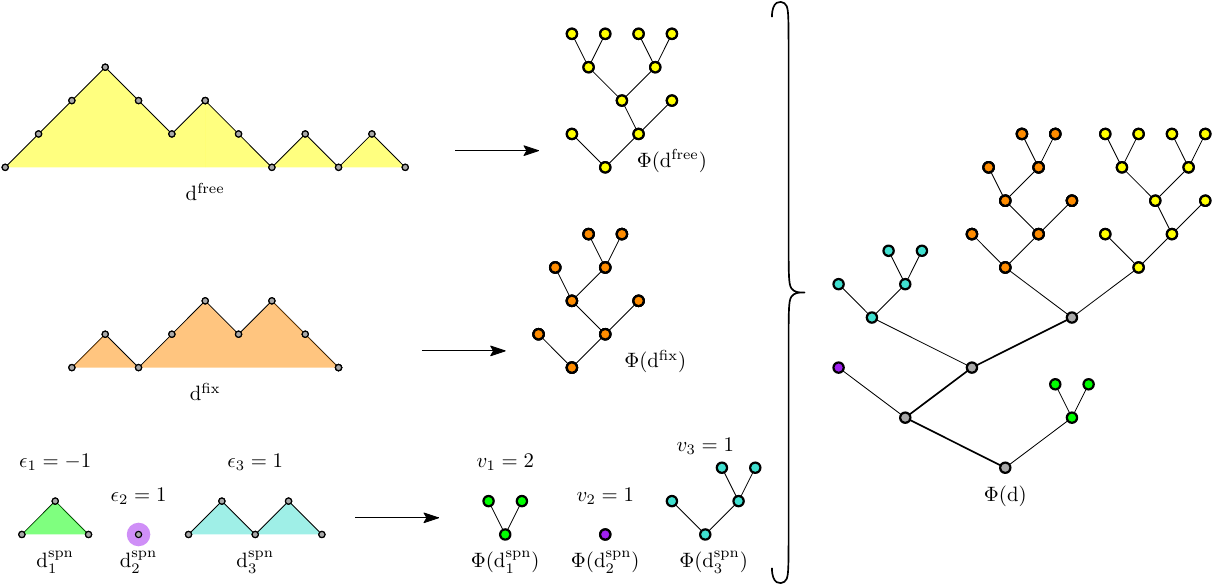}
    \caption{Construction of $\Phi(\path)$, for the path $\path$ of Figure~\ref{fig:sigmas}.
     }
    \label{fig:bijection}
\end{figure}

The following theorem clearly implies Theorem~\ref{thm:main}. 
\begin{thm}\label{thm:main0}
    For all $n,h \in \Z_{\ge 0}$, the function  $\Phi|_{\cD_{n,h}}$ is a bijection from $\cD_{n,h}$ to $\cB_{n,h}$. 
\end{thm}
The proof makes use of two lemmas; the first verifies that the ``size parameter'' $n$ is preserved by $\Phi$, and the second is an analogue of Lemma~\ref{lem_tree:one_step_recursion}, but for the Dyck path decomposition described above.
\begin{lem}\label{lem:size_fixed}
For all $n \in \Z_{\ge 0}$, $\Phi(\cD(n)) \subseteq \cB(n)$. 
\end{lem}
\begin{proof}
    This is true by definition when $n=0$. 
    For $n > 0$, for any Dyck path $\path \in \cD(n)$, we have that $\path^{\mathrm{fix}} \in \cD((\sigma_d-\sigma_g)/2-1)$, that $\path^{\mathrm{free}} \in \cD((2n-(\sigma_{\mathrm{last}}-\sigma_g))/2)$, and that $\path^{\mathrm{spn}}_j \in \cD((\rho_{j+1}-\rho_j)/2-1)$ for $j \in \ints{1,\ell}$. Since $\sum_{j=1}^{\ell} (\rho_{j+1}-\rho_j)=\sigma_{\mathrm{last}}-\sigma_d$, and since the elements of $\{u_1 \ldots u_j : 0\leq j\leq \ell\}$ are all internal vertices of $\Phi(\path)$, it follows by induction that the total number of internal vertices of $\Phi(\path)$ is 
    \begin{align*} 
    \frac{\sigma_d-\sigma_g}2-1 + \frac{2n-(\sigma_{\mathrm{last}}-\sigma_g)}2 + \sum_{j=1}^{\ell}\pran{\frac{\rho_{j+1}-\rho_j}2 - 1} + |\{u_1 \ldots u_j : 0\leq j\leq \ell\}|
    = n\, ,
    \end{align*}
    as required. 
\end{proof}
\begin{lem}\label{lem:one_step_recursion0}
For all $h \ge 1$, the function $F$ given by 
\[
\path\stackrel{F}{\longmapsto}\big (\path^{\mathrm{fix}},\path^{\mathrm{free}},((\epsilon_j(\path),\path_j^{\mathrm{spn}}) ; 1\leq j\leq \ell(\path))\big)
\]
is a bijection from $\mathcal{D}_h$ to the set
\[\cD_h':=\mathcal{D}_{h-m-1}\ \times\ \bigcup_{k=m}^{h-1}\mathcal{D}_k \ \times \ \bigcup_{\ell\ge 0}\Big(\big(\{-1\}\times\bigcup_{k=0}^{\halfceil{h}-1}\mathcal{D}_k\big)\, \cup\, \big(\{1\}\times\bigcup_{k=0}^{\halffloor{h}-1}\mathcal{D}_k\big)\Big)^\ell.\]
\end{lem}

\begin{proof}
    Recall that $m=\lfloor h/2 \rfloor$. The fact that $F(\path) \in \cD_h'$ for all $\path \in \cD_h$ follows from Lemma~\ref{lem:path_obvs}.  
    To see that $F$ is indeed a bijection, we construct its inverse.  Given \[
    \big(\mathtt{d}^{\mathrm{fix}},\mathtt{d}^{\mathrm{free}},((\varepsilon_j,\mathtt{d}_j^{\mathrm{spn}});1\leq j\leq l)\big)\in\cD_h'\, ,
    \]
    it is useful to let $n^{\mathrm{fix}}$, $n^{\mathrm{free}}$, and $n_j$ for $j \in \ints{1,l}$ respectively denote the ``half-lengths'' of $\mathtt{d}^{\mathrm{fix}}$, $\mathtt{d}^{\mathrm{free}}$, and $\mathtt{d}_j^{\mathrm{spn}}$ so that $\mathtt{d}^{\mathrm{fix}} \in \cD(n^{\mathrm{fix}})$ and the like. Now, let $\mathtt{d}^{\mathrm{fix},*}:\ints{1,2n^{\mathrm{fix}}+2}\to \Z_{\ge 0}$ be defined by
    \[
    \mathtt{d}^{\mathrm{fix},*}(i) = 
    \begin{cases}
    \mathtt{d}^{\mathrm{fix}}(i-1)+m+1 
    & \mbox{ if }i \in \ints{1,2n^{\mathrm{fix}}+1}\, ,\\
    m   &\mbox{ if }i = 2n^{\mathrm{fix}}+2\, .
    \end{cases}
    \]
    Next, let $\sigma^*=\sigma^*(\mathtt{d}^{\mathrm{free}})=\max(i \in \ints{0,2n^{\mathrm{free}}}: \mathtt{d}^{\mathrm{free}}(i)=m)$. Then let $\mathtt{d}^{\mathrm{free},1}:\ints{0,\sigma^*}\to \Z_{\ge 0}$ be defined by $\mathtt{d}^{\mathrm{free},1}(i)=\mathtt{d}^{\mathrm{free}}(i)$, and let $\mathtt{d}^{\mathrm{free},2}:\ints{1,2n^{\mathrm{free}}-\sigma^*}\to \Z_{\ge 0}$ be defined by $\mathtt{d}^{\mathrm{free},2}(i)=\mathtt{d}^{\mathrm{free}}(\sigma^*+i)$. 
    Finally, for $1 \le j \le l$, let 
    $\mathtt{d}^{\mathrm{spn},*}_j:\ints{1,2n_j+2} \to \Z_{\ge 0}$ be defined by 
    \[
    \mathtt{d}^{\mathrm{spn},*}_j(i) = 
    \begin{cases}
        m+\varepsilon_j+\varepsilon_j\mathtt{d}^{\mathrm{spn}}_j(i-1) & \mbox{ if }i \in \ints{1,2n_j+1}\\
        m &\mbox{ if }i=2n_j+2\, ,
    \end{cases}
    \]
    Then we observe that the following concatenation of sequences yields a Dyck path $\mathtt{d}$:
    \begin{equation}
    \label{inverse_F}
\mathtt{d}=\mathtt{d}^{\mathrm{free},1}\mathtt{d}^{\mathrm{fix},*}\mathtt{d}^{\mathrm{spn},*}_1\ldots \mathtt{d}^{\mathrm{spn},*}_l \mathtt{d}^{\mathrm{free},2}\, .
    \end{equation}
    Moreover, it holds that $\|\mathtt{d}\|$ is equal to the maximum between $\|\mathtt{d}^{\mathrm{free}}\|$, $m+1+\|\mathtt{d}^{\mathrm{fix}}\|$, and the $m+1+\|\mathtt{d}_j^{\mathrm{spn}}\|$ for $j\in\ints{1,l}$ with $\varepsilon_j=1$. By definition of $\cD_h'$, the expression (\ref{inverse_F}) thus defines a map $\widetilde{F}$ from $\cD_h'$ to $\cD_h$. We now claim that $\widetilde{F}\circ F(\path)=\path$ for any Dyck path $\path\in\cD_h$. Indeed, $\path(\sigma_g)=m$ and $\path$ does not come back to $m$ after $\sigma_{\mathrm{last}}$, by definition, so we have $\sigma^*(\path^{\mathrm{free}})=\sigma_g(\path)$, which then entails the claim by construction. 
    
    Lastly, we want to show that $F\circ\widetilde{F}$ is the identity on $\cD_h'$. By construction, this amounts to proving that $\sigma_g(\mathtt{d})=\sigma^*$, $\sigma_d(\mathtt{d})-\sigma_g(\mathtt{d})-2=2n^{\mathrm{fix}}$, $\ell(\mathtt{d})=l$, and $\rho_{j+1}(\mathtt{d})-\rho_j(\mathtt{d})=2n_j$ for all $j\in\ints{1,l}$.
    From (\ref{inverse_F}) and the definition of $\cD_h'$, it is clear that $\sigma_{\max}(\mathtt{d})=\sigma^*+\min(i\,:\,\mathtt{d}^{\mathrm{fix},*}(i)=h)$. It then follows that $\sigma_d(\mathtt{d})=\sigma^*+2n^{\mathrm{fix}}+2$ and $\sigma_g(\mathtt{d})=\sigma^*$, because $\mathtt{d}^{\mathrm{free}}(\sigma^*)=m$ by definition. Furthermore, $\ell(\mathtt{d})$ is the number of times that the sequences $\mathtt{d}_1^{\mathrm{spn},*},\ldots,\mathtt{d}_l^{\mathrm{spn},*},\mathtt{d}^{\mathrm{free},2}$ hit $m$. Finally, we readily see that $\mathtt{d}^{\mathrm{free},2}$ does not hit $m$, by definition of $\sigma^*$, and that each $\mathtt{d}_j^{\mathrm{spn},*}$ for $j\in\ints{1,l}$ hits $m$ exactly once and at time $2n_j+2$. This concludes the proof.
\end{proof}
\begin{proof}[Proof of Theorem~\ref{thm:main0}]
We will show by induction over $h$ that $\Phi$ is an injective map from $\cD_h$ to $\cB_h$; this holds by definition for $h=0$. Now, fix $h\geq 1$ and suppose that for each $h'<h$, for each $\path\in \cD_{h'}$, $\Phi(\path)\in \cB_{h'}$.
Then writing $m=\lfloor h/2\rfloor$, by Lemma \ref{lem:one_step_recursion0},
\begin{align*}&F(\path)=\big (\path^{\mathrm{fix}},\path^{\mathrm{free}},(\epsilon_j,\path_j^{\mathrm{spn}} ; 1\leq j\leq \ell)\big)\\& \quad \in \mathcal{D}_{h-m-1}\ \times\ \bigcup_{k=m}^{h-1}\mathcal{D}_k \ \times \ \bigcup_{\ell\ge 0}\Big(\big(\{-1\}\times\bigcup_{k=0}^{m-1}\mathcal{D}_k\big)\, \cup\, \big(\{1\}\times\bigcup_{k=0}^{h-m-1}\mathcal{D}_k\big)\Big)^\ell\, .\end{align*} 
Then, by the induction hypothesis, $\Phi(\path^{\mathrm{fix}})\in \cB_{h-m-1}$, $\Phi(\path^{\mathrm{free}})\in \cup_{k=m}^{h-1}\cB_k$, and for each $1\leq j \leq \ell$, if $\epsilon_j=-1$ then  $\Phi(\path^{\mathrm{spn}}_j)\in \cup_{k=0}^{m-1}\cB_k$ and if $\epsilon_j=1$ then  $\Phi(\path^{\mathrm{spn}}_j)\in \cup_{k=0}^{h-m-1}\cB_k$. Moreover, the construction of $\Phi(\path)$ from $\Phi(\path^{\mathrm{fix}})$, $\Phi(\path^{\mathrm{free}})$ and $\Phi(\path^{\mathrm{spn}}_1),\dots, \Phi(\path^{\mathrm{spn}}_\ell)$ 
exactly corresponds to the inverse (\ref{inverse_G}) of the invertible decomposition $G$ of a binary tree with Horton--Strahler number $h$ given in 
Lemma~\ref{lem_tree:one_step_recursion}, which shows that $\Phi(\path)\in \cB_{h}$. Lastly,  as the maps $G$ and $F$ from  Lemmas~\ref{lem_tree:one_step_recursion} and~\ref{lem:one_step_recursion0} are both invertible, the construction of $\Phi(\path)$ from $\path$ is invertible by induction, since the paths 
$\path^{\mathrm{fix}}$, $\path^{\mathrm{free}}$ and $(\path^{\mathrm{spn}}_j,j \in \ints{1,\ell})$ all have height strictly less than $h$. 

The fact that $\Phi$ is in fact a bijection from $\cD_h$ to $\cB_h$ can likewise proved inductively, using that $F$ and $G$ are bijections. Combining this fact with 
Lemma~\ref{lem:size_fixed}, it follows that 
\[
\Phi(\cD_{n,h}) \subset \cB_{n,h}
\]
for all $n,h \in \Z_{\ge 0}$. 
Since the sets $(\cD_{n,h},n \ge 0)$ partition $\cD_h$ and the sets $(\cB_{n,h},n \ge 0)$ partition $\cB_h$, it must then be that in fact $\Phi(\cD_{n,h})=\cB_{n,h}$ for all $n,h \in \Z_{\ge 0}$, which proves the theorem. 
\end{proof}

\bibliographystyle{plainnat}
\bibliography{hs}

\begin{thebibliography}{19}
\providecommand{\natexlab}[1]{#1}
\providecommand{\url}[1]{\texttt{#1}}
\expandafter\ifx\csname urlstyle\endcsname\relax
  \providecommand{\doi}[1]{doi: #1}\else
  \providecommand{\doi}{doi: \begingroup \urlstyle{rm}\Url}\fi

\bibitem[Addario-Berry et~al.(2024)Addario-Berry, Albenque, Donderwinkel, and Khanfir]{US2}
Louigi Addario-Berry, Marie Albenque, Serte Donderwinkel, and Robin Khanfir.
\newblock {Refined Horton-Strahler numbers II: a continuous bijection}.
\newblock \emph{In preparation}, 2024.

\bibitem[Brandenberger et~al.(2021)Brandenberger, Devroye, and Reddad]{BraDevRed21}
Anna Brandenberger, Luc Devroye, and Tommy Reddad.
\newblock {The Horton–Strahler number of conditioned Galton–Watson trees}.
\newblock \emph{Electronic Journal of Probability}, 26:\penalty0 1 -- 29, 2021.
\newblock \doi{10.1214/21-EJP678}.
\newblock URL \url{https://doi.org/10.1214/21-EJP678}.

\bibitem[Burd et~al.(2000)Burd, Waymire, and Winn]{Burd}
Gregory~A. Burd, Edward~C. Waymire, and Ronald~D. Winn.
\newblock {A self-similar invariance of critical binary Galton--Watson trees}.
\newblock \emph{Bernoulli}, 6\penalty0 (1):\penalty0 1 -- 21, 2000.
\newblock \doi{10.2307/3318630}.
\newblock URL \url{https://doi.org/10.2307/3318630}.

\bibitem[Devroye and Kruszewski(1995)]{Dev95}
Luc Devroye and Paul Kruszewski.
\newblock A note on the {H}orton--{S}trahler number for random trees.
\newblock \emph{Information Processing Letters}, 56\penalty0 (2):\penalty0 95 -- 99, 1995.
\newblock ISSN 0020-0190.
\newblock \doi{10.1016/0020-0190(95)00114-R}.
\newblock URL \url{https://doi.org/10.1016/0020-0190(95)00114-R}.

\bibitem[Drmota and Prodinger(2006)]{Drm06}
Michael Drmota and Helmut Prodinger.
\newblock {The register function for $t$-ary trees}.
\newblock \emph{ACM Transactions on Algorithms}, 2\penalty0 (3):\penalty0 318 -- 334, 2006.
\newblock ISSN 1549-6325.
\newblock \doi{10.1145/1159892.1159894}.
\newblock URL \url{https://doi.org/10.1145/1159892.1159894}.

\bibitem[Esparza et~al.()Esparza, Luttenberger, and Schlund]{esparza}
Javier Esparza, Michael Luttenberger, and Maximilian Schlund.
\newblock History of {S}trahler {N}umbers — with a {P}reface.
\newblock International Conference on Language and Automata Theory and Applications, 2014.
\newblock \doi{10.1007/978-3-319-04921-2_1}.
\newblock URL \url{https://doi.org/10.1007/978-3-319-04921-2_1}.

\bibitem[Flajolet et~al.(1979)Flajolet, Raoult, and Vuillemin]{FLAJOLET79}
Philippe Flajolet, Jean-Claude Raoult, and Jean~E. Vuillemin.
\newblock The number of registers required for evaluating arithmetic expressions.
\newblock \emph{Theoretical Computer Science}, 9\penalty0 (1):\penalty0 99 -- 125, 1979.
\newblock ISSN 0304 -- 3975.
\newblock \doi{10.1016/0304-3975(79)90009-4}.
\newblock URL \url{https://doi.org/10.1016/0304-3975(79)90009-4}.

\bibitem[Fran{\c{c}}on(1984)]{franccon1984nombre}
Jean Fran{\c{c}}on.
\newblock Sur le nombre de registres n{\'e}cessaires {\`a} l'{\'e}valuation d'une expression arithm{\'e}tique.
\newblock \emph{RAIRO. Informatique th{\'e}orique}, 18\penalty0 (4):\penalty0 355--364, 1984.
\newblock \doi{10.1051/ita/1984180403551}.
\newblock URL \url{https://doi.org/10.1051/ita/1984180403551}.

\bibitem[{G{\'e}rard Viennot}(2002)]{GERARDVIENNOT2002317}
Xavier {G{\'e}rard Viennot}.
\newblock A {S}trahler bijection between {D}yck paths and planar trees.
\newblock \emph{Discrete Mathematics}, 246\penalty0 (1):\penalty0 317--329, 2002.
\newblock \doi{10.1016/S0012-365X(01)00265-5}.
\newblock URL \url{https://doi.org/10.1016/S0012-365X(01)00265-5}.

\bibitem[Horton(1945)]{Hor45}
Robert~E. Horton.
\newblock Erosional development of streams and their drainage basins ; hydrophysical approach to quantitative morphology.
\newblock \emph{GSA Bulletin}, 56\penalty0 (3):\penalty0 275 -- 370, 1945.
\newblock ISSN 0016-7606.
\newblock \doi{10.1130/0016-7606(1945)56[275:EDOSAT]2.0.CO;2}.
\newblock URL \url{https://doi.org/10.1130/0016-7606(1945)56[275:EDOSAT]2.0.CO;2}.

\bibitem[Kemp(1979)]{KEMP79}
Rainer Kemp.
\newblock The average number of registers needed to evaluate a binary tree optimally.
\newblock \emph{Acta Informatica}, 11:\penalty0 363 -- 372, 1979.
\newblock \doi{10.1007/bf00289094}.
\newblock URL \url{https://doi.org/10.1007/bf00289094}.

\bibitem[Khanfir(2023)]{Kha23}
Robin Khanfir.
\newblock {The Horton--Strahler number of Galton--Watson trees with possibly infinite variance}.
\newblock 2023.
\newblock URL \url{https://arxiv.org/abs/2307.05983}.
\newblock preprint arXiv:2307.05983.

\bibitem[Khanfir(2024)]{Kha24}
Robin Khanfir.
\newblock {Fluctuations of the Horton--Strahler number of stable Galton--Watson trees}.
\newblock 2024.
\newblock URL \url{https://arxiv.org/abs/2401.13771}.
\newblock preprint arXiv:2401.13771.

\bibitem[Kovchegov and Zaliapin(2020)]{KovZal20}
Yevgeniy Kovchegov and Ilya Zaliapin.
\newblock {Random self-similar trees: A mathematical theory of Horton laws}.
\newblock \emph{Probability Surveys}, 17:\penalty0 1 -- 213, 2020.
\newblock \doi{10.1214/19-PS331}.
\newblock URL \url{https://doi.org/10.1214/19-PS331}.

\bibitem[Kovchegov and Zaliapin(2021)]{KovZal21}
Yevgeniy Kovchegov and Ilya Zaliapin.
\newblock {Invariance and attraction properties of Galton–Watson trees}.
\newblock \emph{Bernoulli}, 27\penalty0 (3):\penalty0 1789 -- 1823, 2021.
\newblock \doi{10.3150/20-BEJ1292}.
\newblock URL \url{https://doi.org/10.3150/20-BEJ1292}.

\bibitem[Kovchegov et~al.(2023)Kovchegov, Xu, and Zaliapin]{KovXuZal23}
Yevgeniy Kovchegov, Guochen Xu, and Ilya Zaliapin.
\newblock {Invariant Galton–Watson trees: metric properties and attraction with respect to generalized dynamical pruning}.
\newblock \emph{Advances in Applied Probability}, pages 1 -- 29, 2023.
\newblock \doi{10.1017/apr.2022.39}.
\newblock URL \url{https://doi.org/10.1017/apr.2022.39}.

\bibitem[Strahler(1952)]{Str52}
Arthur~N. Strahler.
\newblock Hypsometric (area-altitude) analysis of erosional topography.
\newblock \emph{GSA Bulletin}, 63\penalty0 (11):\penalty0 1117 -- 1142, 1952.
\newblock ISSN 0016-7606.
\newblock \doi{10.1130/0016-7606(1952)63[1117:HAAOET]2.0.CO;2}.
\newblock URL \url{https://doi.org/10.1130/0016-7606(1952)63[1117:HAAOET]2.0.CO;2}.

\bibitem[Viennot(1990)]{Vie90}
Xavier Viennot.
\newblock Trees.
\newblock In \emph{Mots, mélanges offert à M.P. Schützenberger}. Hermès, Paris, 1990.
\newblock URL \url{https://www.labri.fr/perso/dorbec/seminaire/Trees_Marco.pdf}.

\bibitem[Zeilberger(1990)]{ZEILBERGER199089}
Doron Zeilberger.
\newblock A bijection from ordered trees to binary trees that sends the pruning order to the {S}trahler number.
\newblock \emph{Discrete Mathematics}, 82\penalty0 (1):\penalty0 89--92, 1990.
\newblock \doi{10.1016/0012-365x(90)90047-l}.
\newblock URL \url{https://doi.org/10.1016/0012-365x(90)90047-l}.

\end{thebibliography}

\end{document}